\documentclass[a4paper]{amsart}

\usepackage{amsfonts, latexsym, amssymb, graphicx,
multicol, mathrsfs, color, amscd, verbatim, paralist,
xspace, url, euscript, stmaryrd,  amsmath, amscd, enumitem,
bbold, multirow, tikz, mathtools, mathabx}
\usepackage[all,pdf,cmtip]{xy}

\usepackage[colorlinks, linkcolor=blue, citecolor=magenta, urlcolor=cyan]{hyperref}

\usepackage{mathabx}

\def\ra{\rightarrow}

\def\qed{{\hfill $\Box$}}
\newtheorem{theorem}{THEOREM}[section]

\newtheorem{proposition}[theorem]{Proposition}
\newtheorem{lemma}[theorem]{Lemma}

\newtheorem{example}[theorem]{Example}
\theoremstyle{definition}

\theoremstyle{remark}
\newtheorem{remark}[theorem]{Remark}

\newcommand\CC{{\mathbb C}}

\newcommand\RR{{\mathbb R}}

\newcommand\NN{{\mathbb N}}

\def\SO{\mathop{\rm SO}\nolimits}

\def\Re{\mathop{\rm Re}\nolimits}
\def\Im{\mathop{\rm Im}\nolimits}

\makeatletter
\def\blfootnote{\xdef\@thefnmark{}\@footnotetext}
\makeatother

\begin{document}

\title[On the classification by Morimoto and Nagano]{On the classification by Morimoto and Nagano}
\author[Isaev]{Alexander Isaev}

\address{Mathematical Sciences Institute\\
Australian National University\\
Acton, Canberra, ACT 2601, Australia}
\email{alexander.isaev@anu.edu.au}

\maketitle

\thispagestyle{empty}

\pagestyle{myheadings}

\begin{abstract} 
We consider a family $M_t^3$, with $t>1$, of real hypersurfaces in a complex affine $3$-dimensional quadric arising in connection with the classification of homogeneous compact simply-connected real-analytic hypersurfaces in\, $\CC^n$ due to Morimoto and Nagano. To finalize their classification, one needs to resolve the problem of the CR-embeddability of $M_t^3$ in\, $\CC^3$. In our earlier article we showed that $M_t^3$ is CR-embeddable in\, $\CC^3$ for all $1<t<\sqrt{(2+\sqrt{2})/3}$. In the present paper we prove that $M_t^3$ can be immersed in $\CC^3$ for every $t>1$ by means of a polynomial map. In addition, one of the immersions that we construct helps simplify the proof of the above CR-embeddability theorem and extend it to the larger parameter range $1<t<\sqrt{5}/2$.
\end{abstract}

\section{Introduction}\label{intro}
\setcounter{equation}{0}

This paper concerns the following classical problem investigated by Morimoto and Nagano in \cite{MN}: determine all compact simply-connected real-analytic hypersurfaces in $\CC^n$ homogeneous under an action of a Lie group by CR-transformations. It was shown in \cite{MN} that every such hypersurface is CR-equivalent to either the sphere $S^{2n-1}$ or, for $n=3,7$, to a manifold from the 1-parameter family $M_t^n$ as defined below.

To introduce $M_t^n$ for any $n\ge 2$, consider the $n$-dimensional affine quadric in $\CC^{n+1}$:
\begin{equation}
Q^n:=\{(z_1,\dots,z_{n+1})\in\CC^{n+1}:z_1^2+\dots+z_{n+1}^2=1\}.\label{quadric}
\end{equation}
The group $\SO(n+1,\RR)$ acts on $Q^n$, with the orbits of the action being the sphere $S^n=Q^n\cap\RR^{n+1}$ as well as the compact strongly pseudoconvex hypersurfaces
\begin{equation}
M_t^n:=\{(z_1,\dots,z_{n+1})\in\CC^{n+1}: |z_1|^2+\dots+|z_{n+1}|^2=t\}\cap Q^n,\,\, t>1,\label{mtn}
\end{equation}
which are simply-connected for $n\ge 3$. These hypersurfaces are all nonspherical (see \cite[Remark 2.2]{I1}) and pairwise CR-nonequivalent (see \cite[Example 13.9]{KZ}, \cite[Theorem 2]{BH}). They are the boundaries of Grauert tubes around $S^n$ (note that $Q^n$ can be naturally identified with the tangent bundle $T(S^n)$). In \cite{MN}, Morimoto and Nagano did not investigate the question of whether or not $M_t^n$ admits a real-analytic CR-embedding in $\CC^n$ for $n=3,7$, thus their classification in these two dimensions was not finalized.

The family $M_t^n$ was studied in our papers \cite{I1}, \cite{I2}. In particular, in \cite[Corollary 2.1]{I1} we observed that a necessary condition for the existence of a real-analytic CR-embedding of $M_t^n$ in $\CC^n$ is the embeddability of the sphere $S^n$ in $\CC^n$ as a totally real submanifold. The problem of the existence of a totally real embedding of $S^n$ in $\CC^n$ was considered by Gromov (see \cite{G1} and \cite[p.~193]{G2}), Stout-Zame (see \cite{SZ}), Ahern-Rudin (see \cite{AR}), Forstneri\v c (see \cite{F1}, \cite{F2}, \cite{F3}). In particular, $S^n$ was shown to admit a smooth totally real embedding in $\CC^n$ only for $n=3$, hence $M_t^7$ cannot be real-analytically CR-embedded in $\CC^7$. On the other hand, since $S^3$ is a totally real submanifold of $Q^3$, any real-analytic totally real embedding of $S^3$ in $\CC^3$ (which is known to exist, for instance, by \cite{AR}) extends to a biholomorphic map defined in a neighborhood of $S^3$ in $Q^3$. Owing to the fact that $M_t^3$ accumulate to $S^3$ as $t\ra 1$, this observation implies that $M_t^3$ admits a real-analytic CR-embedding in $\CC^3$ for all $t$ sufficiently close to 1. Thus, the classification of homogeneous compact simply-connected real-analytic hypersurfaces in complex dimension 3 is special as it contains manifolds other than the sphere $S^5$.

More precisely, in \cite[Theorem 1.1]{I2} we showed that $M_t^3$ embeds in $\CC^3$ for all $1<t<\sqrt{(2+\sqrt{2})/3}$ by means of a real-analytic CR-map. This was proved by analyzing the holomorphic continuation of the explicit polynomial totally real embedding of $S^3$ in $\CC^3$ constructed in \cite{AR} (see Remark \ref{othermaps} for details), and the argument was quite involved computationally. Since the hypersurfaces in the family $M_t^3$ are all pairwise CR-nonequivalent, \cite[Theorem 1.1]{I2} leaves the problem of the CR-embeddability of $M_t^3$ in $\CC^3$ for $t\ge \sqrt{(2+\sqrt{2})/3}$ completely open. In this paper, we make steps towards resolving this problem and also look at related matters. 

Even the question of whether every $M_t^3$ can be \emph{immersed} in $\CC^3$ is nontrivial. The answer to this question is positive as the quadric $Q^n$ is known to admit a holomorphic immersion in $\CC^n$ for every $n\ge 2$  (see \cite[p.~19]{BN}). However, such an immersion cannot be algebraic. In our first main result we show that, nevertheless, an algebraic immersion exists for each $M_t^3$:

\begin{theorem}\label{main1}
Every hypersurface $M_t^3$ with $t>1$ can be immersed in $\CC^3$ by means of a polynomial map $\CC^4\to\CC^3$.
\end{theorem}

\noindent In the proof of Theorem \ref{main1} in the next section, for every $n\in\NN$ we explicitly construct a polynomial map $F_n:\CC^4\ra\CC^3$ that yields an immersion of $M_t^3$ for all $1<t<t_n$, where $t_n\ra\infty$ as $n\ra\infty$. None of the maps $F_n$ is injective on $M_t^3$ if $t\ge\sqrt{2}$, thus these maps certainly cannot be used to establish the CR-embeddability of $M_t^3$ in $\CC^3$ for $t\ge \sqrt{2}$ (see Remark \ref{remarkinj}). We did not investigate $F_n$ for injectivity on $M_t^3$ if $\sqrt{(2+\sqrt{2})/3}\le t<\sqrt{2}$ and $n>1$. It is possible that for some $t$ in this range and sufficiently large $n$ the immersions of $M_t^3$ given by the maps $F_n$ are in fact embeddings, but the calculations required to verify the injectivity of $F_n$ on $M_t^3$ for $n>1$ appear to be rather prohibiting. In fact, even studying the fibers of a much simpler polynomial map considered in \cite{I2} was computationally quite hard. 

At the same time, the first map $F_1$ in the sequence $\{F_n\}$ turns out to be easier to handle than its counterpart studied in \cite{I2}, which allows us to improve \cite[Theorem 1.1]{I2} and obtain our second main result: 

\begin{theorem}\label{main2} 
The hypersurface $M_t^3$ admits a real-analytic CR-embeddings in $\CC^3$ if $1<t<\sqrt{5}/2$.
\end{theorem}

\noindent Note that the bound $\sqrt{5}/2$ that appears in Theorem \ref{main2} is only slightly greater than the bound $\sqrt{(2+\sqrt{2})/3}$ of \cite[Theorem 1.1]{I2}. Thus, our purpose here was not so much to enlarge the range of $t$ as to produce a more transparent argument for CR-embeddability, which we were able to achieve by invoking the map $F_1$ instead of the map arising from \cite{AR}. Theorem \ref{main2} is established in Section \ref{sect3}. 

Our construction of the sequence $\{F_n\}$ is to some extent inspired by article \cite{AR}, in which harmonic polynomials were used to produce the map presented in the main theorem therein (see Remark \ref{othermaps}). However, we do not explicitly utilize harmonicity; rather, we directly come up with suitable polynomial immersions. The maps $F_n$ are also of independent interest as each of them yields an explicit totally real embedding of $S^3$ in $\CC^3$. It then follows that each $F_n$ defines a CR-embedding of $M_t^3$ to $\CC^3$ for $1<t<\tau_n$, where $\tau_n\le\sqrt{2}$ (see Remark \ref{embofsn}).   

{\bf Acknowledgment.} We thank L\'aszl\'o Lempert for drawing our attention to reference \cite{BH} and Stefan Nemirovski for many useful discussions. This work was done while the author was visiting the Steklov Mathematical Institute in Moscow. The research is supported by the ARC Discovery Grant DP190100354.

\section{Proof of Theorem \ref{main1}}\label{sect2}
\setcounter{equation}{0}

For convenience, we will argue in the coordinates
$$
w_1:=z_1+iz_2,\,\,w_2:=z_1-iz_2,\,\,w_3:=z_3+iz_4,\,\,w_4:=z_3-iz_4.\label{coordw}
$$
In these coordinates the quadric $Q^3$ becomes
\begin{equation}
\left\{(w_1,w_2,w_3,w_4)\in\CC^4: w_1w_2+w_3w_4=1\right\}\label{qnew}
\end{equation}
(see (\ref{quadric})), the sphere $S^3\subset Q^3$ becomes
\begin{equation}
\left\{(w_1,w_2,w_3,w_4)\in\CC^4: w_2=\bar w_1,\,w_4=\bar w_3\right\}\cap Q^3,\label{sphereq}
\end{equation}
and the hypersurface $M_t^3\subset Q^3$ becomes
\begin{equation}
\left\{(w_1,w_2,w_3,w_4)\in\CC^4: |w_1|^2+|w_2|^2+|w_3|^2+|w_4|^2=2t\right\}\cap Q^3\label{formm7}
\end{equation}
(see (\ref{mtn})).

Let $F:\CC^4\to\CC^3$ be a map of the form
\begin{equation}
(w_1,w_2,w_3,w_4)\mapsto (w_1,w_3,f(w_1,w_2,w_3,w_4)),\label{formF}
\end{equation}
where $f$ is an entire function. Clearly, $F$ yields an immersion of $M_t^3$ in $\CC^3$ if and only if its restriction $\tilde F:=F|_{Q^3}$ is nondegenerate at every point of $M_t^3$. We need the following fact.

\begin{lemma}\label{nondegenerate}
The map $\tilde F$ is nondegenerate at a point $w^0=(w_1^0,w_2^0,w_3^0,w_4^0)\in Q^3$ if and only if one has
\begin{equation}
w_3^0\frac{\partial f}{\partial w_2}(w^0)-w_1^0\frac{\partial f}{\partial w_4}(w^0)\ne 0.\label{nondegencondi}
\end{equation}
\end{lemma}

\begin{proof}
Observe that $|w_1|+|w_3|>0$ on $Q^3$ (see (\ref{qnew})). For $w_1\ne 0$ we choose $w_1,w_3,w_4$ as local coordinates on $Q^3$ and write the third component of $\tilde F$ as
$$
\varphi:=f\left(w_1,\frac{1-w_3w_4}{w_1},w_3,w_4\right).
$$
In this coordinate chart the Jacobian $J_{\tilde F}$ of $\tilde F$ is calculated as
$$
J_{\tilde F}=\frac{\partial\varphi}{\partial w_4}=-\frac{w_3}{w_1}\frac{\partial f}{\partial w_2}+\frac{\partial f}{\partial w_4},\label{phieq}
$$
hence it is nonvanishing at $w^0\in Q^3$ with $w_1^0\ne 0$ if and only if (\ref{nondegencondi}) holds.

Analogously, if $w_3\ne 0$, we choose $w_1,w_2,w_3$ as local coordinates on $Q^3$ and write the third component of $\tilde F$ as
$$
\psi:=f\left(w_1,w_2,w_3,\frac{1-w_1w_2}{w_3}\right).
$$
In this coordinate chart we have
$$
J_{\tilde F}=-\frac{\partial\psi}{\partial w_2}=-\frac{\partial f}{\partial w_2}+\frac{w_1}{w_3}\frac{\partial f}{\partial w_4},
$$
hence $J_{\tilde F}$ does not vanish at $w^0\in Q^3$ with $w_3^0\ne 0$ if and only if condition (\ref{nondegencondi}) is satisfied.\end{proof}

We will now construct a sequence $\{F_n\}_{n\ge 1}$ of maps of the form (\ref{formF}):
\begin{equation}
F_n(w_1,w_2,w_3,w_4)=(w_1,w_3,P_n(w_1,w_2,w_3,w_4)),\label{formoffn}
\end{equation}
where $P_n$ is a polynomial, having the property: there exists $t_n>1$ such that
\begin{equation}
w_3\frac{\partial P_n}{\partial w_2}-w_1\frac{\partial P_n}{\partial w_4}\ne 0\,\,\hbox{on $M_t^3$ for $1<t<t_n$,}\label{condpolynom}
\end{equation}
and
\begin{equation}
t_n\ra\infty\,\,\hbox{as $n\ra\infty$}.\label{tendstoi}
\end{equation}
Lemma \ref{nondegenerate} will then imply the theorem.

Our strategy for coming up with polynomials $P_n$ as above is as follows:
\vspace{0.1cm}

\begin{itemize}

\item[{\bf (A)}] Choose some polynomials $R_n$ having the property: there exists $t_n>1$ such that
\begin{equation}
R_n\ne 0\,\,\hbox{on $M_t^3$ for $1<t<t_n$,}\label{rndonotvanis}
\end{equation}
and (\ref{tendstoi}) holds.
\vspace{0.1cm}

\item[{\bf (B)}] For each $n$ find a polynomial $P_n$ that solves the equation
\begin{equation}
w_3\frac{\partial P_n}{\partial w_2}-w_1\frac{\partial P_n}{\partial w_4}=R_n\,\,\hbox{everywhere on $\CC^4$}.\label{equationforp}
\end{equation}
By (\ref{rndonotvanis}), the polynomials $P_n$ will satisfy (\ref{condpolynom}) as required.

\end{itemize}
\vspace{0.1cm}

Coming up with suitable polynomials $R_n$ in Part (A) such that there are solutions to (\ref{equationforp}) in Part (B) is not easy. After much computational experimentation, we discovered that polynomials of the following form work well:
\begin{equation}
R_n:=\left(\left(\frac{1}{2}+ia_n\right)w_1w_2-\left(\frac{1}{2}-ia_n\right)w_3w_4\right)^{2n},\label{defrn}
\end{equation}
where $a_n$ are positive numbers to be chosen later. 

Let us study the zeroes of $R_n$ on $M_t^3$. First of all, we restrict $R_n$ to $Q^3$ by replacing $w_1w_2$ with $1-w_3w_4$:
\begin{equation}
R_n\big|_{Q^3}=\left(w_3w_4-\left(\frac{1}{2}+ia_n\right)\right)^{2n}\label{restrrn}
\end{equation}
(see (\ref{qnew})). Thus, $R_n$ vanishes at a point $w=(w_1,w_2,w_3,w_4)\in Q^3$ if and only if
$$
w_1w_2=\frac{1}{2}-ia_n,\quad w_3w_4=\frac{1}{2}+ia_n.
$$
Such a point lies in $M_t^3$ for some $t>1$ if and only if
$$
|w_1|^2+\left|\frac{1}{2}-ia_n\right|^2\frac{1}{|w_1|^2}+|w_3|^2+\left|\frac{1}{2}+ia_n\right|^2\frac{1}{|w_3|^2}=2t
$$
(see (\ref{formm7})), or, equivalently, if and only if
\begin{equation}
|w_1|^2+\left(\frac{1}{4}+a_n^2\right)\frac{1}{|w_1|^2}+|w_3|^2+\left(\frac{1}{4}+a_n^2\right)\frac{1}{|w_3|^2}=2t.\label{idd8}
\end{equation}

The following elementary lemma, which we state without proof, will often be helpful:
 
\begin{lemma}\label{min}
For fixed $p>0$, let
$$
g(x):=x+\frac{p}{x},\quad x>0.
$$
Then $\min_{x>0} g(x)=2\sqrt{p}$.
\end{lemma}

Letting in Lemma \ref{min}
$$
p=\frac{1}{4}+a_n^2,\label{choicep}
$$
from (\ref{idd8}) we see
$$
t\ge2\sqrt{\frac{1}{4}+a_n^2}.
$$
We then set
\begin{equation}
t_n:=2\sqrt{\frac{1}{4}+a_n^2}.\label{deftn}
\end{equation}
Clearly, with this choice of $t_n$ condition (\ref{rndonotvanis}) holds. To satisfy (\ref{tendstoi}), the positive numbers $a_n$ will be chosen to increase to $\infty$.

\begin{remark}\label{degenerationbeyond}
Notice that the map $\tilde F_n:=F_n\big|_{Q^3}$ degenerates at some point of $M_t^3$ whenever $t\ge t_n$. Indeed, for any such $t$ find $x_0>0$ satisfying $g(x_0)=t$. Then the point
$$
\left(\sqrt{x_0},\frac{\frac{1}{2}-ia_n}{\sqrt{x_0}},\sqrt{x_0},\frac{\frac{1}{2}+ia_n}{\sqrt{x_0}}\right)
$$
lies in $M_t^3$, and $\tilde F_n$ is degenerate at it.
\end{remark}    

Let us now turn to Part (B). We fix $n\ge 1$ and look for a solution to equation (\ref{equationforp}) in the form
\begin{equation}
P_n=\sum_{k=1}^{2n}\alpha_kw_1^{2n-k}w_2^{2n-k+1}w_3^{k-1}w_4^k,\label{formofpn}
\end{equation}
where $\alpha_j$ are complex numbers, which will be computed in terms of $a_n$ shortly. With $P_n$ given by (\ref{formofpn}), the left-hand side of (\ref{equationforp}) is
\begin{equation}
\begin{array}{l}
\displaystyle -\alpha_1(w_1w_2)^{2n}+\alpha_{2n}(w_3w_4)^{2n}+\\
\vspace{-0.3cm}\\
\displaystyle\hspace{1cm}\sum_{k=1}^{2n-1}\Bigl((k+1)\alpha_{2n-k}-(2n-k+1)\alpha_{2n-k+1}\Bigr)(w_1w_2)^k(w_3w_4)^{2n-k}.\label{left}
\end{array}
\end{equation}

Write the right-hand side of (\ref{equationforp}) as
\begin{equation}
\begin{array}{l}
\displaystyle\left(\frac{1}{2}+ia_n\right)^{2n}(w_1w_2)^{2n}+\left(\frac{1}{2}-ia_n\right)^{2n}(w_3w_4)^{2n}+\\
\vspace{-0.1cm}\\
\hspace{0.8cm}\displaystyle\sum_{k=1}^{2n-1}(-1)^k{{2n}\choose{k}}  \left(\frac{1}{2}+ia_n\right)^k\left(\frac{1}{2}-ia_n\right)^{2n-k}(w_1w_2)^k(w_3w_4)^{2n-k}.\label{right}
\end{array}
\end{equation}
Comparing (\ref{left}) and (\ref{right}), we, first of all, see
\begin{equation}
\alpha_1=-\left(\frac{1}{2}+ia_n\right)^{2n},\quad \alpha_{2n}=\left(\frac{1}{2}-ia_n\right)^{2n}.\label{indbasis}
\end{equation}
Next, using the expression for $\alpha_{2n}$ from (\ref{indbasis}), comparison of (\ref{left}) and (\ref{right}) for $k$ increasing from 1 to $n-1$ yields by induction
$$
\alpha_{2n-k}=\frac{1}{k+1}{{2n}\choose{k}}\sum_{\ell=0}^k(-1)^{\ell}\left(\frac{1}{2}+ia_n\right)^{\ell}\left(\frac{1}{2}-ia_n\right)^{2n-\ell},\,\,k=1,\dots,n-1,
$$
in particular
\begin{equation}
\alpha_{n+1}=\frac{1}{n}{{2n}\choose{n-1}}\sum_{\ell=0}^{n-1}(-1)^{\ell}\left(\frac{1}{2}+ia_n\right)^{\ell}\left(\frac{1}{2}-ia_n\right)^{2n-\ell}.\label{alphanplus1}
\end{equation}

Similarly, using the expression for $\alpha_1$ from (\ref{indbasis}), comparison of (\ref{left}) and (\ref{right}) for $k$ decreasing from $2n-1$ to $n+1$ yields by induction
$$
\begin{array}{l}
\displaystyle\alpha_{2n-k+1}=-\frac{1}{2n-k+1}{{2n}\choose{k}}\sum_{\ell=0}^{2n-k}(-1)^{\ell}\left(\frac{1}{2}+ia_n\right)^{2n-\ell}\left(\frac{1}{2}-ia_n\right)^{\ell},\\
\vspace{-0.3cm}\\
\displaystyle \hspace{9cm}k=n+1,\dots,2n-1,
\end{array}
$$
in particular
\begin{equation}
\alpha_n=-\frac{1}{n}{{2n}\choose{n+1}}\sum_{\ell=0}^{n-1}(-1)^{\ell}\left(\frac{1}{2}+ia_n\right)^{2n-\ell}\left(\frac{1}{2}-ia_n\right)^{\ell}.\label{alphan}
\end{equation}

We will now compare the remaining terms of (\ref{left}) and (\ref{right}), namely, the ones corresponding to $k=n$:
$$
(n+1)(\alpha_n-\alpha_{n+1})=(-1)^n{{2n}\choose{n}}  \left(\frac{1}{2}+ia_n\right)^n\left(\frac{1}{2}-ia_n\right)^{n}.
$$
Invoking formulas (\ref{alphanplus1}), (\ref{alphan}) then leads to the following condition on $a_n$:
\begin{equation}
\sum_{\ell=0}^{2n}(-1)^{\ell}\left(\frac{1}{2}+ia_n\right)^{\ell}\left(\frac{1}{2}-ia_n\right)^{2n-\ell}=0.\label{condan}
\end{equation}
Dividing by $\left(\frac{1}{2}-ia_n\right)^{2n}$ and setting
$$
A_n:=-\frac{\frac{1}{2}+ia_n}{\frac{1}{2}-ia_n},
$$
we see that (\ref{condan}) is equivalent to 
$$
\sum_{\ell=0}^{2n}A_n^{\ell}=0.
$$
Summing up the first $2n+1$ terms of the geometric series with common ratio $A$, we then obtain
$$
1-A^{2n+1}=0,
$$
or, equivalently, 
\begin{equation}
\Re\left(\frac{1}{2}+ia_n\right)^{2n+1}=0.\label{conddeterminean}
\end{equation}

Thus, $a_n$ must have the property that the value of the function $z^{2n+1}$ at\linebreak $\frac{1}{2}+ia_n$ is an imaginary number. Certainly, $z^{2n+1}$ takes imaginary values on the line $\Re z=1/2$, and we let $\frac{1}{2}+ia_n$ be the point with the largest possible argument at which $\arg z^{2n+1}=\pi/2$. Namely, we choose $a_n>0$ so that
\begin{equation}
\arg\left(\frac{1}{2}+ia_n\right)=\frac{\frac{\pi}{2}+2\pi K}{2n+1},\label{choicean}
\end{equation}       
where $K$ is the largest integer less than $n/2$. It is then clear that (\ref{conddeterminean}) holds and 
$$
\arg\left(\frac{1}{2}+ia_n\right)\ra\frac{\pi}{2}\,\,\hbox{as $n\ra\infty$},
$$
hence the sequence $\{a_n\}$ converges to $\infty$. Therefore, by (\ref{deftn}), condition (\ref{tendstoi}) is satisfied.

We have thus found two sequences of polynomials $\{R_n\}$ and $\{P_n\}$ as required in Parts (A) and (B). The proof of the theorem is complete.\qed
\vspace{0.1cm}

To give a better idea of the above argument, we will now write out details for $n=1$. This special case will be required for our proof of Theorem \ref{main2} in the next section.

\begin{example}\label{exn=1} \rm Let $n=1$. Condition (\ref{choicean}) yields
$$
a_1=\frac{1}{2\sqrt{3}}.
$$ 
Hence, by (\ref{defrn}) we have
$$
R_1=\left(\frac{1}{6}+\frac{i}{2\sqrt{3}}\right)(w_1w_2)^2-\frac{2}{3}w_1w_2w_3w_4+\left(\frac{1}{6}-\frac{i}{2\sqrt{3}}\right)(w_3w_4)^2,
$$
and (\ref{formofpn}), (\ref{indbasis}) lead to a formula for $P_1$:
\begin{equation}
P_1=-\left(\frac{1}{6}+\frac{i}{2\sqrt{3}}\right)w_1w_2^2w_4+\left(\frac{1}{6}-\frac{i}{2\sqrt{3}}\right)w_2w_3w_4^2.\label{formp1}
\end{equation}

Further, by (\ref{deftn}), we have
\begin{equation}
t_1=\frac{2}{\sqrt{3}}.\label{t1}
\end{equation}
Then, by Remark \ref{degenerationbeyond}, the map $\tilde F_1$ is nondegenerate at every point of $M_t^3$ if and only if $1<t<2/\sqrt{3}$. We will investigate $\tilde F_1$ for injectivity in the next section.  
\end{example}

\section{Proof of Theorem \ref{main2}}\label{sect3}
\setcounter{equation}{0}

In order to produce a CR-embedding of $M_t^3$ in $\CC^3$ for $1<t<\sqrt{5}/2$ we will utilize the map $\tilde F_1$. It is clear from  (\ref{t1}) that $\sqrt{5}/2<t_1$, thus we only need to show that $\tilde F_1$ is injective on each $M_t^3$ for $t$ in this range. 

We start by studying the fibers of $\tilde F_1$. 

\begin{proposition}\label{propfibers}
Let two points $w=(w_1,w_2,w_3,w_4)$ and $\hat w=(\hat w_1,\hat w_2,\hat w_3,\hat w_4)$ lie in $Q^3$ and assume that $\tilde F_1(w)=\tilde F_1(\hat w)$. Then we have:
\begin{itemize}
\item[{\rm (a)}] $\hat w_1=w_1$, $\hat w_3=w_3$;

\item[{\rm (b)}] if $w_1=0$ or $w_3=0$, then $\hat w=w$;

\item[{\rm (c)}] if $w_1\ne 0$ and $w_3\ne 0$, then either $\hat w=w$ or one of the following holds:
\begin{equation}
\begin{array}{l}
\displaystyle\hat w_4=-\frac{1+i\sqrt{3}}{2w_3}(w_3w_4-1),\\
\vspace{-0.1cm}\\
\displaystyle\hat w_4=-\frac{1-i\sqrt{3}}{2w_3}\left(w_3w_4-\frac{1+i\sqrt{3}}{2}\right);
\end{array}\label{solquadr}
\end{equation}

\item[{\rm (d)}] neither of the two values in the right-hand side of {\rm (\ref{solquadr})} is equal to $w_4$ if $w\in M_t^3$ for $t<t_1$;

\item[{\rm (e)}] the two values in the right-hand side of {\rm (\ref{solquadr})} are distinct if $w\in M_t^3$ for $t<t_1$.

\end{itemize}
Hence, the fiber $\tilde F_1^{-1}(\tilde F_1(w))$ consists of at most three points, and, if $w\in M_t^3$ with $w_1\ne 0$, $w_3\ne 0$ for $t<t_1$, it consists of exactly three points.

\end{proposition}

\begin{proof} Part (a) is immediate from (\ref{formoffn}). Furthermore, (\ref{formp1}) yields
\begin{equation}
\begin{array}{l}
\displaystyle-\left(\frac{1}{6}+\frac{i}{2\sqrt{3}}\right)w_1\hat w_2^2\hat w_4+\left(\frac{1}{6}-\frac{i}{2\sqrt{3}}\right)\hat w_2w_3\hat w_4^2=\\
\vspace{-0.1cm}\\
\hspace{3cm}\displaystyle-\left(\frac{1}{6}+\frac{i}{2\sqrt{3}}\right)w_1w_2^2w_4+\left(\frac{1}{6}-\frac{i}{2\sqrt{3}}\right)w_2w_3w_4^2,\label{mainids}
\end{array}
\end{equation}
which together with (\ref{qnew}) implies part (b).    

From now on, we assume that $w_1\ne 0$ and $w_3\ne 0$. Then using (\ref{qnew}) we substitute
\begin{equation}
w_2=\frac{1-w_3w_4}{w_1},\quad \hat w_2=\frac{1-w_3\hat w_4}{w_1}\label{expr234}
\end{equation}
into (\ref{mainids}) and simplifying the resulting expression obtain
\begin{equation}
\begin{array}{l}
\displaystyle(\hat w_4-w_4)\left[-\frac{w_3^2}{3}\hat w_4^2+\left(\left(\frac{1}{2}+\frac{i}{2\sqrt{3}}\right)w_3-\frac{w_3^2w_4}{3}\right)\hat w_4+\right.\\
\vspace{-0.3cm}\\
\hspace{3cm}\displaystyle\left.\left(\frac{1}{2}+\frac{i}{2\sqrt{3}}\right)w_3w_4-\frac{w_3^2w_4^2}{3}-\left(\frac{1}{6}+\frac{i}{2\sqrt{3}}\right)\right]=0.
\end{array}\label{mainids4}
\end{equation}
We treat identity (\ref{mainids4}) as an equation with respect to $\hat w_4$. By part (a) and formula (\ref{expr234}), the solution $\hat w_4=w_4$ leads to the point $w$. Further, the solutions of the quadratic equation given by setting the expression in the square brackets in (\ref{mainids4}) to zero, are shown in formula (\ref{solquadr}). This establishes part (c). 

Next, for $\hat w_4=w_4$ the expression in the square brackets in (\ref{mainids4}) becomes
$$
-\left(w_3w_4-\left(\frac{1}{2}+\frac{i}{2\sqrt{3}}\right)\right)^2=-R_1\big|_{Q^3}
$$
(cf.~(\ref{restrrn})), which implies that $\tilde F_1$ is degenerate at $w$. This contradicts our choice of $w$ and thus establishes part (d).

Finally, the two values in (\ref{solquadr}) are easily seen to be equal if and only if
$$
w_3w_4=\frac{1}{2}+\frac{i}{2\sqrt{3}},
$$
which leads to the same contradiction as in part (d), so part (e) follows. The proof is complete.

\end{proof} 

Proposition \ref{propfibers} implies that in order to establish Theorem \ref{main2} it suffices to show that for every value $1<t<\sqrt{5}/2$ and every point
$$
w=\left(w_1,\frac{1-w_3 w_4}{w_1},w_3,w_4\right)\in M_t^3\,\,\hbox{with $w_1\ne 0$, $w_3\ne 0$}, 
$$
the point $\hat w:=\left(w_1,(1-w_3\hat w_4)/w_1,w_3,\hat w_4\right)$ does not lie in $M_t^3$ for any of the two choices of $\hat w_4$ in (\ref{solquadr}). In fact, we only need to consider the first solution in (\ref{solquadr}) as the second solution turns into the first one upon interchanging $w_4$ and $\hat w_4$.

Let now $\hat w$ correspond to the first choice of $\hat w_4$ in (\ref{solquadr}) and assume that $\hat w\in M_t^3$. Set $b:=w_3w_4$, $\hat b:=w_3\hat w_4$. Then by (\ref{solquadr}) we see
\begin{equation}
\hat b=-\frac{1+i\sqrt{3}}{2}(b-1).\label{solquadr1}
\end{equation}

We have
\begin{equation}
\begin{array}{l}
\displaystyle |w_1|^2+\frac{|b-1|^2}{|w_1|^2}+|w_3|^2+\frac{|b|^2}{|w_3|^2}=2t,\\
\vspace{-0.1cm}\\
\displaystyle |w_1|^2+\frac{|\hat b-1|^2}{|w_1|^2}+|w_3|^2+\frac{|\hat b|^2}{|w_3|^2}=2t,
\end{array}\label{condonab}
\end{equation}
where by (\ref{solquadr1}) the second equation can be rewritten as
\begin{equation}
|w_1|^2+\left|b-\displaystyle\frac{1+i\sqrt{3}}{2}\right|^2\frac{1}{|w_1|^2}+|w_3|^2+\frac{|b-1|^2}{|w_3|^2}=2t.\label{condonab1}
\end{equation}
By Lemma \ref{min} it then follows that the point $b$ lies in the intersection of the interiors of two ellipses:
$$
{\mathcal D}:={\mathcal E}_1\cap{\mathcal E}_2,
$$
where
$$
{\mathcal E}_1:=\left\{z\in\CC: |z-1|+|z|<\frac{\sqrt{5}}{2}\right\},
$$
with $\partial {\mathcal E}_1$ having foci at 1 and 0, and
$$
{\mathcal E}_2:=\left\{z\in\CC: |z-1|+\left|z-\displaystyle\frac{1+i\sqrt{3}}{2}\right|<\frac{\sqrt{5}}{2}\right\},
$$ 
with $\partial {\mathcal E}_2$ having foci at 1 and $(1+i\sqrt{3})/2$.

Subtracting (\ref{condonab1}) from the first equation in (\ref{condonab}) yields
\begin{equation}
\frac{b_1-\sqrt{3}b_2}{|w_1|^2}+\frac{1-2b_1}{|w_3|^2}=0,\label{subtract}
\end{equation}
where $b_1:=\Re b$ and $b_2:=\Im b$. Observe that neither of the numerators in (\ref{subtract}) is zero since neither of the lines $2b_1-1=0$, $b_1-\sqrt{3}b_2=0$ intersects ${\mathcal D}$. 

By (\ref{subtract}) we have
$$
|w_1|^2=\frac{b_1-\sqrt{3}b_2}{2b_1-1}|w_3|^2.
$$
Plugging this expression into the first identity in (\ref{condonab}) and simplifying the resulting formulas, we obtain
$$
\left(\frac{b_1-\sqrt{3}b_2}{2b_1-1}+1\right)|w_3|^2+\left(\frac{|b-1|^2(2b_1-1)}{(b_1-\sqrt{3}b_2)}+|b|^2\right)\frac{1}{|w_3|^2}=2t.
$$
Lemma \ref{min} then yields
\begin{equation}
\left(\frac{b_1-\sqrt{3}b_2}{2b_1-1}+1\right)\left(\frac{|b-1|^2(2b_1-1)}{(b_1-\sqrt{3}b_2)}+|b|^2\right)<\frac{5}{4}.\label{theinequality}
\end{equation}

Let us denote the left-hand side of (\ref{theinequality}) by $\phi(b)$. We will now study the behavior of the function $\phi$ in the domain ${\mathcal D}$ and prove:

\begin{lemma}\label{philarge}
One has
$$
\phi(b)\ge \frac{5}{4}
$$
for all $b\in{\mathcal D}$.
\end{lemma}

\begin{proof}
Let ${\mathcal L}$ be the line $b_1+\sqrt{3}b_2-1=0$, which we write in parametric form as
$$
\frac{-3+i\sqrt{3}}{8}\sigma+1,\quad\sigma\in\RR.
$$
The segment ${\mathcal I}:={\mathcal L}\cap{\mathcal D}$ is defined by the parameter range
\begin{equation}
-\frac{1}{\sqrt{15}+3}<\sigma<\frac{1}{\sqrt{15}-3}.\label{paramrange}
\end{equation}
It passes through the common focus of $\partial {\mathcal E}_1$ and $\partial {\mathcal E}_2$ at 1 (for $\sigma=0$) and its closure joins the two points of the intersection $\partial {\mathcal E}_1\cap\partial {\mathcal E}_2$. By restricting $\phi$ to ${\mathcal I}$, one obtains the quadratic function
$$
\hat\phi(\sigma):=\frac{1}{4}(3\sigma^2-6\sigma+8),\label{hatphi}
$$
which is easily seen to be greater than or equal to $5/4$ everywhere.

Next, we will restrict $\phi$ to line segments orthogonal to ${\mathcal L}$ and lying in ${\mathcal D}$. Fix $\sigma_0\in{\mathcal I}$ and consider the line ${\mathcal L}_{\sigma_0}$ given in parametric form as
$$
\frac{1+i\sqrt{3}}{2}\tau+\frac{-3+i\sqrt{3}}{8}\sigma_0+1,\quad\tau\in\RR.
$$  
Clearly, ${\mathcal L}_{\sigma_0}$ passes through the point
$$ 
{\mathbf b}_0:=\frac{-3+i\sqrt{3}}{8}\sigma_0+1\in{\mathcal I}
$$
and is orthogonal to ${\mathcal L}$. Restricting $\phi$ to the segment ${\mathcal L}_{\sigma_0}\cap{\mathcal D}$, one obtains the function
\begin{equation}
\hat\phi_{\sigma_0}(\tau):=\frac{4-3\sigma_0}{4(\left(4-3\sigma_0\right)^2-16\tau^2)}\left(\left(32-48\sigma_0\right)\tau^2-9\sigma_0^3+30\sigma_0^2-48\sigma_0+32\right)\label{restrorth}
\end{equation}
for some range of the parameter $\tau$. Notice that in (\ref{restrorth}) the denominator vanishes exactly at the two points where ${\mathcal L}_{\sigma_0}$ intersects the lines $2b_1-1=0$, $b_1-\sqrt{3}b_2=0$; these two points lie outside of ${\mathcal D}$. 

It is straightforward to see from (\ref{restrorth}) that $\hat\phi_{\sigma_0}$ takes its minimum at $\tau=0$, i.e., at the point ${\mathbf b}_0\in{\mathcal I}$. Indeed, one computes
$$
\hat\phi_{\sigma_0}'(\tau)=\frac{32(4-3\sigma_0)h(\sigma_0)}{(\left(4-3\sigma_0\right)^2-16\tau^2)^2}\tau,
$$
where
$$
h(\sigma_0):=-9\sigma_0^3+30\sigma_0^2-36\sigma_0+16.\label{derivrestrorth}
$$
Let us show that $h(\sigma_0)>0$. On $\bar{\mathcal I}$ one has
$$
h'(\sigma)=-27\sigma^2+60\sigma-36<0.
$$
Therefore, $h(\sigma_0)$ is bounded from below by the value of $h$ at the right endpoint of $\bar{\mathcal I}$, i.e., at $1/(\sqrt{15}-3)$ (see (\ref{paramrange})). One easily computes that $h(1)=918\sqrt{15}-3555>0$, thus $h(\sigma_0)>0$ as required. 

Since the function $\hat\phi$ is already known to be greater than or equal to $5/4$, it follows that $\hat\phi_{\sigma_0}$ is greater than or equal to $5/4$ everywhere. 

This completes the proof of the lemma. \end{proof}

As Lemma \ref{philarge} contradicts inequality (\ref{theinequality}), the theorem follows.\qed

\section{A few remarks}\label{sect4}
\setcounter{equation}{0}

We conclude the paper by making a number of observations and comments.

\begin{remark}\label{remarkinj}
First of all, it is easy to see that $\tilde F_n$ is not injective on $M_t^3$ for every $t\ge\sqrt{2}$ and every $n$. Indeed, fix $t\ge\sqrt{2}$, let $u\ne 0$ be a real number satisfying
$$
2u^2+\frac{1}{u^2}=2t, 
$$
(cf.~Lemma \ref{min}) and consider the following two distinct points in $Q^3$:
\begin{equation}
w:=\left(u,\frac{1}{u},u,0\right), w':=\left(u,0,u,\frac{1}{u}\right).\label{threepoints}
\end{equation}
Then $w, w'\in M_t^3$, and it follows from (\ref{formoffn}), (\ref{formofpn}) that $\tilde F_n(w)=\tilde F_n(w')=(u,u,0)$. 
\end{remark}

\begin{remark}\label{embofsn}
The polynomials $R_n$ constructed in the proof of Theorem \ref{main1} do not vanish on the sphere $S^3\subset Q^3$ (see (\ref{sphereq})). Therefore, by (\ref{formoffn}), (\ref{equationforp}) and Lemma \ref{nondegenerate}, all the maps $\tilde F_n$ are nondegenerate at every point of $S^3$. It is also clear that all $\tilde F_n$ are injective on $S^3$. Hence, each $\tilde F_n$ yields an explicit real-analytic totally real embedding of $S^3$ to $\CC^3$. It then follows that each $\tilde F_n$ is biholomorphic is a neighborhood of $S^3$ in $Q^3$ and thus defines a CR-embedding of $M_t^3$ in $\CC^3$ if $1<t<\tau_n$ for some $\tau_n$. We did not attempt to determine or estimate $\tau_n$ for $n>1$ as the calculations involved appear to be quite hard. Recall that, by Remark \ref{remarkinj}, the map $\tilde F_n$ fails to be injective on $M_t^3$ for all $t\ge\sqrt{2}$, so we have $\tau_n\le\sqrt{2}$.       
\end{remark}

\begin{remark}\label{othermaps}
Article \cite{AR} yields a class of maps of the form (\ref{formF}), with the restriction of $f$ to $S^3$ being a harmonic polynomial given by
$$
{\mathcal P}=\left(\bar w_1\frac{\partial}{\partial w_3}-\bar w_3\frac{\partial}{\partial w_1}\right)\left(\sum_{j=1}^m\frac{{\mathcal Q}_j}{p_j(q_j+1)}\right),
$$
where ${\mathcal Q}_j$ is a homogeneous harmonic complex-valued polynomial in $w_1$, $\bar w_1$, $w_3$, $\bar w_3$ of total degree $p_j\ge 1$ in $w_1,w_3$ and total degree $q_j$ in $\bar w_1$, $\bar w_3$, such that the sum ${\mathcal Q}:={\mathcal Q}_1+\dots+{\mathcal Q}_m$ does not vanish on $S^3$. Every map of this kind defines a totally real embedding of $S^3$ to $\CC^3$ and therefore a CR-embedding of $M_t^3$ for $t$ sufficiently close to 1. One can homogenize ${\mathcal P}$ by multiplying its lower-degree homogeneous components by suitable powers of the polynomial $|w_1|^2+|w_3|^2$, which is equal to 1 on $S^3$. The resulting polynomial $\hat{\mathcal P}(w_1,\bar w_1,w_3,\bar w_3)$ may no longer be harmonic; however, the map
$$
(w_1,w_3)\mapsto (w_1,w_3,\hat{\mathcal P}(w_1,\bar w_1,w_3,\bar w_3))
$$
still defines the same totally real embedding of $S^3$ to $\CC^3$ and its extension to $\CC^4$
\begin{equation}
(w_1,w_2,w_3,w_4)\mapsto (w_1,w_3, \hat{\mathcal P}(w_1,w_2,w_3,w_4))\label{armaps}
\end{equation}
the same CR-embedding of $M_t^3$ for $t$ sufficiently close to 1.

As one example, in \cite{AR} the authors set
$$
\begin{array}{l}
\displaystyle{\mathcal P}=w_3\bar w_1\bar w_3^2-w_1\bar w_1^2\bar w_3+i\bar w_1\bar w_3=\\
\vspace{-0.3cm}\\
\displaystyle\hspace{1.8cm}\left(\bar w_1\frac{\partial}{\partial w_3}-\bar w_3\frac{\partial}{\partial w_1}\right)\left(\frac{|w_1|^4-4|w_1|^2|w_3|^2+|w_3|^4}{6}+\frac{i(|w_3|^2-|w_1|^2)}{2}\right).
\end{array}
$$
Here
\begin{equation}
{\mathcal Q}=|w_1|^4-4|w_1|^2|w_3|^2+|w_3|^4+i(|w_3|^2-|w_1|^2),\label{examq}
\end{equation}
and it is not hard to see that ${\mathcal Q}$ indeed does not vanish on $S^3$. To homogenize ${\mathcal P}$, one multiplies its lowest-degree homogeneous component $i\bar w_1\bar w_3$ by $|w_1|^2+|w_3|^2$, which yields the polynomial
$$
\hat{\mathcal P}(w_1,\bar w_1,w_3,\bar w_3)=(1+i)(w_3\bar w_1\bar w_3^2+iw_1\bar w_1^2\bar w_3).
$$
This is the polynomial (up to the factor $1+i$) that appears in the main theorem of \cite{AR}. Its natural extension to $\CC^4$ is
\begin{equation}
\hat{\mathcal P}(w_1,w_2,w_3,w_4)=(1+i)(w_2w_3w_4^2+iw_1 w_2^2 w_4).\label{polyphatexam}
\end{equation}
In \cite{I1}, \cite{I2} we investigated the corresponding map (\ref{armaps}) for nondegeneracy and injectivity and eventually proved in \cite[Theorem 1.1]{I2} that this map yields a CR-embedding of $M_t^3$ to $\CC^3$ for all $1<t<\sqrt{(2+\sqrt{2})/3}$. Most of our effort went into establishing injectivity for $t$ in this range.

The polynomials $P_n$ that we utilized in the proof of Theorem \ref{main1} in Section \ref{sect2} (see formula (\ref{formofpn})) are homogeneous by construction, and, except in the case $n=1$, we do not know whether they arise from suitable harmonic polynomials by the homogenization procedure described above. For $n=1$ calculations are easy, and $P_1$ is readily seen to come from the inhomogeneous harmonic polynomial
$$
-\frac{1}{6}w_3\bar w_1\bar w_3^2+\frac{1}{6}w_1\bar w_1^2\bar w_3-\frac{i}{2\sqrt{3}}\bar w_1\bar w_3.
$$

Note that for the polynomial $\hat{\mathcal P}$ from (\ref{polyphatexam}), the expression 
$$
w_3\frac{\partial \hat{\mathcal P}}{\partial w_2}-w_1\frac{\partial \hat{\mathcal P}}{\partial w_4},
$$
when restricted to $Q^3$, has two distinct roots if regarded as a function of the product $w_3w_4$ (see \cite[formulas (2.7), (2.8)]{I2}). For comparison, from (\ref{equationforp}) we see that the analogous expression for $P_n$ in place of $\hat{\mathcal P}$ is equal to $R_n$ whose restriction to $Q^3$ has only one (multiple) root (see (\ref{restrrn})). This makes our polynomials $P_n$ easier to deal with in computations. One illustration of this is the proof of Theorem \ref{main2}, where we used $F_1$ instead of map (\ref{armaps}) with $\hat{\mathcal P}$ given by (\ref{polyphatexam}), on which the proof of \cite[Teorem 1.1]{I2} was based. In particular, formulas (\ref{solquadr}) are less complicated than the corresponding formulas in \cite{I2}. Overall, the proof of Theorem \ref{main2} is computationally much more transparent than that of \cite[Theorem 1.1]{I2}.

It is possible that one can investigate the CR-embeddability of $M_t^3$ in $\CC^3$ for all $t$ by using other maps of the form (\ref{armaps}). Note, however, that while it is tempting to take ${\mathcal Q}$ to be a polynomial in $|w_1|^2$, $|w_3|^2$ (as was done in (\ref{examq})), one should avoid doing so as otherwise map (\ref{armaps}) will not be injective on $M_t^3$ with $t\ge\sqrt{2}$. This follows exactly as in Remark \ref{remarkinj}; namely (\ref{armaps}) takes equal values at the two points defined in (\ref{threepoints}).
\end{remark}

\end{document}